\documentclass[12pt]{article}

\usepackage{latexsym}
\usepackage{amssymb}
\usepackage{amsmath}
\usepackage{enumerate}
\usepackage{amsmath, amsthm,amsfonts,amssymb,latexsym}
\usepackage[all]{xy}
\usepackage{enumerate}
\DeclareMathOperator{\RE}{Re}

\DeclareMathOperator{\tr}{tr}
\DeclareMathOperator{\Tr}{Tr}
\newcommand{\ceqv}{\overset{cyc}{\sim}}
\newcommand{\deqv}{ \overset{d}{\sim} }

\newcommand{\A}{\mathcal{A}}

\newcommand{\F}{\mathcal{F}}
\newcommand{\Fg}{\mathbb{F}}

\newcommand{\Real}{\mathbb{R}}
\newcommand{\Complex}{\mathbb{C}}

\newcommand{\abs}[1]{\left\vert#1\right\vert}
\newcommand{\set}[1]{\left\{#1\right\}}
\newcommand{\seq}[1]{\left<#1\right>}
\newcommand{\norm}[1]{\left\Vert#1\right\Vert}

\newtheorem{theorem}{Theorem}
\newtheorem*{theorem1}{Theorem}
\newtheorem{lemma}[theorem]{Lemma}
\newtheorem{proposition}[theorem]{Proposition}
\newtheorem{corollary}[theorem]{Corollary}
\newtheorem{definition}[theorem]{Definition}

\newtheorem{remark}[theorem]{Remark}

\begin{document}
\title{Algebraic reformulation of Connes
embedding problem and the free group algebra.}
\author{Kate Juschenko, Stanislav Popovych}
\date{}
\maketitle

 \footnotetext{ 
(Secondary) }

\begin{abstract}
We give a modification of I. Klep and  M. Schweighofer algebraic
reformulation of Connes' embedding problem by considering
$*$-algebra of the countably generated free group. This allows to
consider only quadratic polynomials in unitary generators instead of
arbitrary polynomials in self-adjoint generators.

\end{abstract}

\section{Introduction.}

Let $\omega \in \beta(\mathbb{N})\setminus \mathbb{N}$ be a free
ultrafilter on $\mathbb{N}$ and $R$ be the hyperfinite II$_1$-factor
with faithful tracial normal state $\tau$. Then the subset
$I_\omega$ in $l^\infty(\mathbb{N}, R)$ consisting of
$(x_1,x_2,\ldots)$ with $\lim_{n\to\omega} \tau(x^*_n x_n) =0$ is a
closed ideal in $l^\infty(\mathbb{N}, R)$ and a quotient algebra
$R^\omega =l^\infty(\mathbb{N}, R)/ I_\omega $ is a von Neumann
II$_1$-factor called {\it ultrapower} of $R$. It is naturally
endowed with a faithful tracial normal state
$$\tau_\omega((x_n)+I_\omega) = \lim_{n\to\omega} \tau(x_n).$$

A. Connes' embedding problem asks whether  every finite von Neumann
algebra with fixed normal faithful tracial state can be embedded
into $R_\omega$ in a trace-preserving way.

It is well know that Connes' embedding problem is equivalent to the
problem whether every finite set $x_1, \ldots , x_n$ of self-adjoint
contractions in arbitrary II$_1$-factor $(M, \tau)$ has {\it
matricial microstates}, i.e  whether for any $\varepsilon>0$  and
$t\ge 1$ there is $k\in \mathbb{N}$ and  self-adjoint contractive
$k\times k$-matrices $A_1, \ldots, A_n$ such that $|\tr(w(x_1,
\ldots, x_n))- \tau(w(A_1, \ldots, A_n))|<\varepsilon$ for all words
$w$ of length at most $t$.

In \cite{Hadwin} D. Hadwin proved that solving  Connes' embedding
problem in affirmative is equivalent to proving that  there is no
polynomial $p(x_1,...,x_n)$ in non-commutative variables such that
\begin{enumerate}
\item $tr_k(p(A_1,\ldots,A_n))\geq 0$ for every $k$
and self-adjoint contractions \\ $A_1,...,A_n\in
M_k$.
\item $\tau(p(T_1,\ldots,T_n))<0$, where $T_1,\ldots,T_n$ are
self-adjoint contractive elements in a finite factor with trace
$\tau$.
\end{enumerate}

Recently I. Klep and  M. Schweighofer established that
 Connes' embedding problem has the following equivalent
 algebraic reformulation.

 Let $f(X_1,\ldots, X_m)$ be a self-adjoint
  element in a free associative algebra $\mathbb{K}\langle \overline{X}\rangle$
  with countable family of self-adjoint generators
 $\overline{X}=\set{X_1,X_2,\ldots}$, where $\mathbb{K}=\mathbb{R}$ or $K=\mathbb{C}$. If
 $\tr (f(A_1, \ldots, A_m))\ge 0$ for any $n$ and family of self-adjoint
 contractive matrices $A_1, \ldots, A_m \in M_n(\mathbb{K})$ then $f$ has
 the property that for every $\varepsilon
 >0$ we have
$\varepsilon e +f = g+ c$ where $c$ is a sum of commutators in
$\mathbb{K}\langle \overline{X}\rangle$, $g$ belongs to quadratic
module generated by $1-X_i^2$ and $e$ is the unit in
$\mathbb{K}\langle \overline{X}\rangle$. Recall that a {\it
quadratic module} is the smallest subset   of $\mathbb{K}\langle
\overline{X}\rangle$ containing unit, closed under addition and
conjugation $x\to g^* x g$ by arbitrary $g\in \mathbb{K}\langle
\overline{X}\rangle$.

In the present paper we consider the group $*$-algebra $\F$  of the
countably generated free group $\Fg_\infty=\langle u_1, u_2, \ldots
\rangle$ instead of $\mathbb{K}\langle \overline{X}\rangle$. One
reason is that we can use a more standard and well known  set of
hermitian squares $\set{g^*g| g\in \F}$ instead of quadratic module
$M$ and the second that we can  bound  the degree of polynomials $f$
in the above reformulation by 2. This modification provides the
following.

\begin{theorem1}
Connes' embedding conjecture is true iff for any self-adjoint
$f\in\F$ of the form $f(u_1,\ldots, u_n)= \alpha e + \sum_{i\not= j
}\alpha_{ij} u_i^* u_j$ condition \begin{equation} \label{trp}
          Tr (f(V_1,\ldots, V_n))\ge 0
          \end{equation}
for every $m\ge 1$ and every $n$-tuple of
    unitary   matrices $V_1, \ldots, V_n \in U(m)$ implies that for every $\varepsilon>0$,
    $\varepsilon e + f= g+c$ where  $c$ is a sum of commutators and  $g$ is a
    sum of Hermitian squares.
\end{theorem1}
We will call $f$ satisfying \eqref{trp}  a {\it trace-positive quadratic polynomial}.
Elements of the form $g+c$ with $c$ being a sum of commutators are called
{\it cyclically equivalent} to $g$ (see Section \ref{secalg}).

In Section \ref{seccov} we study a  subset of correlation matrices of
the form $ [\tr(U_i^* U_j)]_{ij}$ where $U_1, \ldots, U_n$ runs over
$n$-tuple of unitary matrices and $\tr(U)$ denotes normalized trace
of $U$. Using Clifford algebra methods we show that this set
contains all correlation matrices with real coefficients. This
implies that all trace-positive quadratic polynomials $f$ with real
coefficients do satisfy the property from the above theorem.

The description of the set $ \set{[\tr(U_i^* U_j)]_{ij} \mid  U_1,
\ldots, U_n \in U_m(\Complex), m\ge 1}$ seems to be unknown even for
$n= 3$. In this case it is equivalent to the problem of description
of the set of triples $(\tr(U), \tr(V), \tr(U V))$ where $U$ and $V$
are unitary matrices. Note that the  lists of possible eigenvalues
of $U$, $V$ and  $U V$ can be described by generalization of Horn's
inequalities (see \cite{Fulton}) but little is know about possible
traces $(\tr(U), \tr(V), \tr(U V))$. The only known connection
between these traces seems to be the inequality $\sqrt{1-\abs{\tr(U
V)}^2} \le \sqrt{1-\abs{\tr(U)}^2} +
 \sqrt{1-\abs{\tr(V)}^2}$ established in \cite{wang}.

\bigskip

\noindent {\bf Acknowledgements}

 The authors are indebted to Igor Klep and Markus Schweighofer for careful reading of the paper and a
number of suggestions which helped to improve the paper a lot.

\section{An algebraic reformulation of Connes' problem.}\label{secalg}

Let $\F$ be the $*$-algebra of the countably generated free group
$\Fg_\infty$. Let $K$ denote the $\mathbb{R}$-subspace in $\F_{sa}$
generated by the commutators $fg-gf$ ($f, g\in F$). We will say that
$f$ and $g$ in $\F$ are {\it cyclically equivalent} (denote $f\ceqv
g$) if $f-g\in K$. Let $\Sigma^2(\F)$ denote the set of positive
elements of the $*$-algebra $\F$, i.e. elements of the form
$\sum_{j=1}^m f_j^* f_j$ with $f_j\in \F$. An element of the form
$f^*f$ is called Hermitian square and therefore the  cone
$\Sigma^2(\F)$ is called the cone of Hermitian squares.

\begin{definition}
Let $C$ be a subset of the vector space $V$. An element $v\in C$ is called an
algebraic interior point of $C$ if for every $u\in V$ there is $\varepsilon >0$ in
$\mathbb{R}$ s.t. $v+ \lambda u \in C$ for all $0\le \lambda \le \varepsilon$.
\end{definition}

\begin{definition}
Let $A$ be a unital $*$-algebra with the unit $e$. Then

\noindent 1.  An element  $a \in {\A}_{sa}$ is called  bounded if
there is $\alpha \in \mathbb{R}_+$ such that $\alpha e \pm a \in
\Sigma^2(\A)$.\\ \noindent 2. An element  $x=a+ib$ with $a,\ b \in
{\A}_{sa}$ is bounded if
 the elements $a$ and  $b$ are such.

\noindent 3.   The algebra ${\A}$ is bounded if all its elements are bounded.

\end{definition}

It is well known that the set of all bounded elements in  $\A$ is a
$*$-subalgebra in $\A$ and that  an element  $x\in {A}$ is bounded
if and only if  $xx^*$ is such (see for example~\cite{pop1, popjush}). In
particular $\F$ is a bounded $*$-algebra. Obviously this implies
that the unit of the algebra is an algebraic interior  point of
$\Sigma^2(\F)$.

The following lemma is a modification of Theorem 3.12
in~\cite{Klep}.

\begin{lemma}\label{cycl}
Let $f\in \F$ be self-adjoint. If for any II$_1$ factor $M$ with
faithful normal tracial state $\tau$ and separable predual and every
$n$-tuple of unitary elements $U_1, \ldots, U_n$ in the unitary
group $\mathcal{U}(M)$ of $M$ we have that
$$\tau(f(U_1,\ldots, U_n))\ge 0$$ then for every $\varepsilon > 0$,
$\varepsilon e +f \sim g$ for
some $g\in \Sigma^2(\F)$.
\end{lemma}
\begin{proof}
Clearly $\Sigma^2(\F)+K$ is a convex cone in $\mathbb{R}$-space $\F_{sa}$.
Since $e$  is an an algebraic internal point of
 $\Sigma^2(\F)$ it is also an algebraic internal point of $\Sigma^2(\F)+K$.

Assume that  there is $\varepsilon > 0$ such that  $\varepsilon e +f
\not\sim g$ for any  $g\in \Sigma^2(\F)$, i.e. $\varepsilon e +f\not
\in \Sigma^2(\F)+K$. By Eidelheit-Kakutani separation theorem there
is $\mathbb{R}$-linear unital functional $L_0:\F_{sa}\to \mathbb{R}$
s.t. $L_0(\Sigma^2(\F)+K) \subseteq \mathbb{R}_{\ge0}$ and
$L_0(\varepsilon e +f)\in \mathbb{R}_{\le 0}$. Since $-K \subset
\Sigma^2(\F)+K$ we have that $L_0(K)=0$. In particular extending
$L_0$ to $\mathbb{C}$-linear functional on $\F$ we get a tracial
functional $L$. Since $L$ maps $\Sigma^2(\F)$ into the non-negative
reals it defines a pre-Hilbert space structure on $\F$ by means of
sesquilinear for $\langle p, q \rangle = L(q^*p)$, $p, q \in \F$.
Let $N= \{p: \langle p, p \rangle=0\}$. By  Cauchy-Schwarz
inequality $N=  \{p: L(q^* p) =0 \text{ for all } q\in \F\}$ and
hence is a left ideal. Let
 $H_0$ be the pre-Hilbert space $\F/N$. Consider the left regular representation $\pi: \F \to L(H_0)$.
 Since $\pi$ is a $*$-homomorphism for every $f\in \F$ operator $\pi(f)$ is bounded as a linear combination of unitary
 operators. Thus $\pi(f)$ can be extended to the bounded operator acting  on the Hilbert space  $H$ which is the completion of $H_0$. Thus we have a representation  $\pi:\F\to B(H)$ with a cyclic
 vector $\xi = e+N$ and such that $L(p) = \langle \pi(p)\xi, \xi \rangle$. In particular
 $L$ is a contractive tracial state on $\F$ and thus defines a tracial state
 of the universal enveloping $C^*$-algebra $C^*(\F)$. By Banach-Alaoglu and Krein-Milman theorem we can
 assume that $L$ is an extreme point in the set of all tracial states
 and thus $\pi(\F)$ generates a factor von Neumann algebra $M$ (see~\cite{Hadwin}).
 Clearly $M$ is a finite factor. If it is type $I$ then it should be $\mathbb{C}$ (since
 $\xi$ is a trace vector) and thus can be embedded into any II$_1$-factor in trace
 preserving way. Thus we can assume that $M$ is a type II$_1$-factor. But then
 condition  $L(f)<0$ is impossible.
\end{proof}

\begin{corollary}\label{realcor}
If self-adjoint $f\in \F$ has real coefficients and
 for any real type II$_1$ von Neumann algebra $(M,\tau)$ with normal faithful tracial state $\tau$
and every $n$-tuple of unitary elements $U_1, \ldots, U_n$ in $M$
we have that
$$\tau(f(U_1,\ldots, U_n))\ge 0$$ then the same holds for the complex II$_1$ von Neumann algebras.
\end{corollary}
\begin{proof}
Element $f$ can be written as $f = \alpha +  \sum_{w_j} \alpha_{w_j} (w_j + w_j^*)$ with
$\alpha_{w_j}\in \Real$ and for
complex trace $\tau$ and $U_1, \ldots, U_n \in U(M)$ we will have
$\tau(f) = \alpha +  2 \sum_{w_j} \alpha_{w_j} \RE  \tau(w_j)$, i.e. $\tau(f) = (\RE \tau)(f)$.
 To finish the proof note that $M$ can be regarded as a real finite von Neumann algebra with faithful trace
 $\RE \tau$.
\end{proof}

\begin{lemma}\label{reallemma}
If $f\in \Real [ \Fg_\infty ] $, $f=f^*$ and for every real type
II$_1$ von Neumann algebra $(M, \tau)$ we have that $\tau(f)\ge 0$
then for every $\varepsilon>0$, $\varepsilon+ f\ceqv g$ for some $g
\in \set{\sum_{j=1}^{m} g_j^* g_j \mid m\in \mathbb{N}, g_j\in \Real
\langle \Fg_\infty\rangle}$.
\end{lemma}
\begin{proof}
The  proof of this statement can be obtained by obvious modification
of the proof of lemma~\ref{cycl}. The only nontrivial part is that the unit $e$ is an algebraic internal point but this is equivalent to
$\Real \langle \Fg_\infty\rangle$ being bounded $*$-algebra. The proof of the last fact can be
found in~\cite{vidav}.
\end{proof}
This lemma gives another proof of corollary~\ref{realcor}. In sequel
we will need  the following lemma.
\begin{lemma}\label{matrarb}
If $(M, \tau)$ is a $II_1$ factor which can be embedded  into
$R^\omega$ and $f\in \F$ is self-adjoint then the condition
$\tr(f(V_1, \ldots, V_n))\ge 0$ for all $m\ge 0$ and all unitary
$V_1, \ldots, V_n$ in $M_{m\times m}(\mathbb{C})$ implies that
$\tau(f(U_1, \ldots, U_n))\ge 0$ for all unitary $U_1, \ldots, U_n$
in $M$.
\end{lemma}
\begin{proof}
Considering $M$  as a subalgebra in $R^\omega$ and $\tau$ as a
restriction of the trace on $R^\omega$ we can find a representing
sequences $\set{u^{(k)}_j}_{j=1}^\infty$ for $U_k$,  $k=1,\ldots, n$
in $l^\infty(\mathbb{N}, R)$ which are unitary elements in von
Neumann algebra $l^\infty(\mathbb{N}, R)$. This can be done since
every unitary in von Neumann algebra $R^\omega$ can be lifted to a
unitary in von Neumann algebra $l^\infty(\mathbb{N}, R)$ with
respect to canonical morphism $\pi:l^\infty(\mathbb{N}, R) \to
R^\omega$.  Taking $j$ sufficiently large we can approximate mixed
moments of $U_1, \ldots, U_k$  up to order $m$, i.e. $\tau(U_{s_1}
\ldots U_{s_t})$ with $t\le m$ and $s_1, \ldots, s_t \in
\set{1,\ldots, n}$, by the mixed moments of unitary matrices
$u^{(k)}_1, \ldots, u^{(k)}_n$.
\end{proof}

The following theorem is Proposition 4.6 in \cite{Kirchberg}
\begin{theorem}{\bf (E. Kirchberg)}
Let ($M$, $\tau$)  be von Neumann algebra with separable predual and
faithful normal tracial state $\tau$. If for all $n\ge 1$ and for
all unitaries  $u_1, \ldots, u_n$ in $M$ and for arbitrary
$\varepsilon
>0$ there exists $m\ge 1$ and unitary $m\times m$ matrices $V_1,
\ldots, V_n \in U(m)$ s.t. for all $i, j$:
\begin{eqnarray}
|\tau(u_i^* u_j) -\frac{1}{m} \Tr(V_i^* V_j)|&<&\varepsilon, \label{c1}\\
|\tau( u_j) -\frac{1}{m} \Tr(V_j)|&<&\varepsilon\label{c2}
\end{eqnarray}
then $M$ can be embedded into $R^\omega$.
\end{theorem}
\begin{remark}
We may drop condition \eqref{c2} since we may take $u_0 =1, u_1,
\ldots, u_n$ and by \eqref{c1} find matrices $W_0, \ldots, W_n$ such
that $|\tau(u_i^* u_j) - \frac{1}{m} \Tr(W_i^* W_j)|<\varepsilon$
for all $i$ and $j$. Thus~\eqref{c1} and \eqref{c2} will be
satisfied if we take $V_j = W_0^*W_j$.
\end{remark}
The proof of the following theorem is an adaptation of the proof of
Proposition 3.17 from \cite{Klep}.
\begin{theorem}\label{embeds}
Let $(M, \tau)$ be $II_1$-factor with separable predual. If for every self-adjoint element $f\in \F$
  of the form  $f= \alpha + \sum_{i\not= j } \alpha_{ij}u_i^* u_j$ the condition
  $$ \Tr (f(V_1,\ldots, V_n))\ge 0 $$ for all $m\ge 1$ and every $n$-tuple of
    unitary   matrices $V_1, \ldots, V_n \in U(m)$ implies that $\tau(f(U_1, \ldots, U_n))\ge 0$
    for all unitaries $U_1, \ldots, U_n$ in $M$ then $M$ can be embedded into $R^\omega$.
\end{theorem}
\begin{proof}
Take  $n\ge 1$. Consider the finite dimensional vector space
 $W= \{ \alpha e + \sum_{i\not= j } \alpha_{ij} u_i^* u_j |
\alpha_{ij}\in \mathbb{C}\}$. Denote by $C$ the convex hull  of the
set $F$ of the functionals $T\in W^*$ of the form $T(p) =
\frac{1}{m} \Tr (p(V_1, \ldots, V_n))$ where $m\ge 1$ and $V_1,
\ldots, V_n \in U(m)$. Take arbitrary  $n$-tuple of unitary elements
$U_1, \ldots, U_n$ in $M$ and put $L(p) = \tau(p(U_1, \ldots, U_n))$
for $p\in W$. Assume that $L\not\in C$. By Hahn-Banach theorem there
is $f\in W^{**}= W$ and $c\in \mathbb{R}$ s.t. $Re(L(f))<c<
Re(T(f))$ for all $T\in C$. Since $e \in W$ we can substitute $f-c$
instead of $f$ and thus assume that $c=0$. Since $T(f^*) =
\overline{T(f)}$ for every $T\in C$ and $L(f^*)=\overline{L(f)}$ we
have that $L(f+f^*)= 2 Re (L(f))<0< 2 Re(T(f))= T(f+f^*)$ which is a
contradiction. Thus $L\in C$. Let $T$ be a rational convex
combination of elements $T_1, \ldots, T_s$ from $F$ and $T_k$
corresponds to $n$-tuples $V_{j,k}$. Then $T =\frac{1}{q}(p_1
T_1+\ldots +p_s T_s)$ for some positive integers $p_1, \ldots, p_s,
q$. Taking block-diagonal $V_j = (V_{j,1}^{\otimes p_1}\oplus \ldots
\oplus V_{j,s}^{\otimes p_s})$ we see that $T\in F$. Thus each
element of $C$, in particular element $L$ can be approximated by
elements of $F$. By the Kirchberg's Theorem we have that $M$ can be
embedded into $R^{\omega}$.

\end{proof}

\begin{theorem}
Connes' embedding conjecture problem has affirmative solution iff
for any self-adjoint $f\in\F$ of the form $f= \alpha e +
\sum_{i\not= j }\alpha_{ij} u_i^* u_j$ condition $$ \Tr
(f(V_1,\ldots, V_n))\ge 0 $$ for every $m\ge 1$ and every $n$-tuple
of
    unitary   matrices $V_1, \ldots, V_n \in U(m)$ implies that for every $\varepsilon>0$,
    $\varepsilon e + f\sim g$ with $g\in \Sigma^2(\F)$.
\end{theorem}
\begin{proof}  If Connes' embedding problem has affirmative solution
and quadratic $f\in\F_{sa}$ is such that $ \Tr (f(V_1,\ldots,
V_n))\ge 0 $ for every $m\ge 1$ and every $n$-tuple of
    unitary   matrices $V_1, \ldots, V_n \in U(m)$ then by lemma \ref{matrarb} we have  $ \tau (f(U_1,\ldots, U_n))\ge 0 $ for any unitary $U_1, \ldots, U_n$ in $M$. Hence by lemma \ref{cycl}, $\varepsilon e + f$ is cyclically equivalent to a sum of Hermitian squares. This proves that the conditions of the theorem are necessary.

If  $\varepsilon e + f$ is cyclically equivalent to an element in  $\Sigma^2(\F)$ for every $\varepsilon>0$ then clearly $ \tau (f(U_1,\ldots, U_n))\ge 0 $ for any unitary $U_1, \ldots, U_n$ in $M$. Hence the sufficiency of the theorem conditions follows from Theorem \ref{embeds}.

\end{proof}

\section{The trace-positive quadratic polynomials.}\label{seccov}

The results of the preceding section motivate the study of
trace-positive self-adjoint quadratic polynomials $f= \alpha e+
\sum_{i\not= j }\alpha_{ij} u_i^* u_j$ in unitary generators $u_1,
\ldots, u_n$, i.e. polynomials having the property that $ \Tr
(f(V_1,\ldots, V_n))\ge 0 $ for every $m\ge 1$ and every $n$-tuple
of unitary matrices $V_1, \ldots, V_n \in U(m)$. If $A$ denotes the
matrix
$$\left(
                                                                            \begin{array}{cccc}
                                                                              \alpha/n & \alpha_{12} & \ldots &  \alpha_{1n} \\
                                                                              \overline{\alpha_{12}} & \alpha/n & \ldots  & \alpha_{2n} \\
                                                                               \ldots &  \ldots&  \ldots &  \ldots \\
                                                                              \overline{\alpha_{1n}}& \overline{\alpha_{2n}} & \ldots  & \alpha/n \\
                                                                            \end{array}
                                                                          \right)
$$
 then $\Tr f(U_1, \ldots, U_n)\ge 0 $ can be expressed
 as  positivity of the sum of all entries of the Schur product $A\circ
 X$  where $X =[\tr (U_i^*U_j)]_{ij}$.

Thus the trace-positive polynomials   $f$ can be characterized as
those for which the sum of all entries of  $A\circ X$ for all $X\in
K_{n}:= \{ [\tr(U_i^*U_j)]_{ij} \mid m \ge 1, U_1, \ldots, U_n \in
U(m) \}$. Thus our primary objective is to describe the sets $K_n
\subseteq M_n(\mathbb{C})$. Note that in the case $A$ is positive
semidefinite we have $f\in \Sigma^2(\F)$. Indeed in this case $A$ is
a sum of rank one positive semidefinite matrices $A =\sum_{s}
(\beta_{s,1},\ldots, \beta_{sn})^T(\beta_{s,1},\ldots, \beta_{sn})$
and hence $f= \sum_{s} (\sum_j \beta_{s,j} u_j)^*(\sum_j \beta_{s,j}
u_j)$. We will also be interested in real analog of the sets $K_n$,
i.e. the sets $K_n(\Real) = K_n \cap M_n(\Real)$. Note that the sets
of the traces of monomials of unitary operators and their asymptotic
properties in the context of Connes' embedding problem also studied
in \cite{radulescu} and \cite{radulescu1}.

A self-adjoint matrix $A$ such that $f = (u_1^{-1}, \ldots,
u_n^{-1}) A (u_1, \ldots, u_n)^T$ is defined uniquely except for the
diagonal entries. This motivates the following definition. We will
call $A$ and $B$ {\it diagonally equivalent} and write $A \deqv B$
if $A-B$ is a diagonal matrix with vanishing trace.

\begin{definition}
Let $S\subseteq  M_n(\Complex)$ and $A\in M_n(\Complex)$ be self-adjoint. We say that $A$ is $S$-positive and denote
$A\ge_S 0$ if  there is self-adjoint $B$ such that $A\deqv B$ and $$\sum_{ij} b_{ij} s_{ij} \ge 0$$
for all $s\in S$.
\end{definition}
The three natural choices for $S$ will be
$$F_n=\set{(t_{ij})| t_{jj} = 1 \text{ and } |t_{ij}|\le 1 \text{ for all } i,j },$$
$P_n \subset F_n$ consisting of positive matrices and the set
$K_n\subset F_n$. Clearly, an self-adjoint matrix  $A=[a_{ij}]$ is
$K_n$-positive iff $f= \sum_i a_{ii} e + \sum_{i\not=j} a_{ij} u_i^*
u_j$ is a trace positive quadratic polynomial.  Note that if $$A
\ge_{F_n} 0$$ then $$\Tr A \ge \sum_{i\not= j} \abs{a_{ij}}$$ and
hence $A\deqv B$ for some diagonally dominant matrix $B$. In this
case polynomial $f = (u_1^{-1}, \ldots, u_n^{-1}) A (u_1, \ldots,
u_n)^T$ is a sum of hermitian squares. However if $A\ge_{P_n} 0$
then  $ A $ need not be diagonally equivalent to positive matrix.
Note that for the three choices of $S$ mentioned above one can use
equality instead of diagonal equivalence since diagonal entries of
elements in $S$ equal to $1$.

The following lemma gives a description of cyclically equivalent
quadratic polynomials.

\begin{lemma} For every matrix $A$ the element
$ (u_1^{-1}, \ldots, u_n^{-1}) A (u_1, \ldots, u_n)^T $ is cyclically equivalent to

\begin{equation}\label{eqviv}
\sum_k g_k^{-1} (u_1^{-1}, \ldots, u_n^{-1}) A_{g_k} (u_1, \ldots, u_n)^T g_k
\end{equation} for any finite
collection $g_1, \ldots, g_k \in \Fg_\infty $ and any matrices $A_g$
such that
\begin{equation}\label{sum}
\sum_k A_{g_k} \deqv A.
\end{equation}

Any element $g\in \F$ such that  $g\ceqv f$ is of the form
\eqref{eqviv} for some matrices satisfying \eqref{sum}. Moreover for
self-adjoint $g$ matrices $A_g$ can also be chosen to be
self-adjoint.
\end{lemma}
\begin{proof}
The lemma follows from the following easy observation. For any $w_1$
and $w_2$ in $\Fg_\infty$ the element $w_1-w_2$ is a commutator
$ab-ba$ for some $a$, $b\in \Fg_\infty$ if and only if  $w_1$ and
$w_2$ are conjugated.  Hence $K$ consists of finitely supported sums
of the form $$\sum_j \sum_k \alpha_{jk} g_k^{-1} w_j g_k$$ where
$w_j$, $g_k$ belong to $\Fg_\infty$ and $\sum_k \alpha_{jk} =0 $ for
all $j$.
\end{proof}

\section{The Clifford Algebras and positive polynomials with real coefficients.}
For a real Hilbert space $V$ there is a unique associative algebra
$\mathcal{C}(V)$ with a linear embedding $J: V\to \mathcal{C}(V)$
with generating range and such that for all $x, y\in V$
\begin{equation}
J(x) J(y) +J(y) J(x) = 2 \seq{x,y}.
\end{equation}

The algebra $\mathcal{C}(V)$ is called Clifford algebras associated
to $V$. Clifford algebra can be realized on a Hilbert space such
that for every $x \in V$ with $\norm{x} = 1$ operator $J(x)$ is
symmetry, i.e.  $J(x)^* = J(x)$ and $J(x)^2=I$. To see this consider
Pauli matrices $$U = \left(
           \begin{array}{cc}
             1 & 0 \\
             0 & -1 \\
           \end{array}
         \right), Q \left(
           \begin{array}{cc}
             0 & 1 \\
             1 & 0 \\
           \end{array}
         \right).
$$ Clearly $U$ and $Q$ are self-adjoint unitary matrices and
 $U^2=I, Q^2 =I$, $Q U +U Q =0$. Then matrices
 $Q_j = U\otimes \ldots \otimes U \otimes Q\otimes I \otimes I \ldots
 $ are symmetries and $\set{Q_i, Q_j} = 2 \delta_{ij} I$. Hence operator
 $J(x)= \sum_j x_j Q_j$ is also a symmetry for unit real vector $x$.
 For further properties of Clifford algebras we refer to the books \cite{BR} and \cite{pisier-book}.

 \begin{theorem}
 For every real correlation matrix $P \in M_n(\Real)$ there is
 $n$-tuple of symmetries $S_1, \ldots, S_n$ in finite dimensional real Hilbert space
   s.t. $P = [ \tr(S_i^* S_j) ]_{ij}$.
 \end{theorem}
 \begin{proof}
 Every correlation $n\times n$-matrix $P$ is a Gram matrix for a system of unit
 vectors $x_1, \ldots, x_n$, i.e. $P=[\seq{x_i, x_j}]_{ij}$.
Taking  Clifford symmetries $S_j = J(x_j)$ as in the paragraph
preceding the theorem we see that $P = [ \tr(S_i^* S_j)]_{ij}$.
\end{proof}

\begin{proposition}
For every $n\ge 1$ the closure $T_n(\Real)$ of the set of matrices
   $$\set{ [\tau(U_i^* U_j)]_{ij} | U_1, \ldots, U_n \in
   \mathcal{U}(M)}$$ does not depend on real type II$_1$ von Neumann algebra $(M,   \tau)$.

  If self-adjoint $f(u_1, \ldots, u_n) \in \F$ has real coefficients and possess property that
  for every $n$-tuple of unitary matrices $U_1, \ldots, U_n \in U(m)$ we have
   $tr(f(U_1, \ldots, U_n))\ge 0$ then for every
$\varepsilon>0$, $\varepsilon e+ f\ceqv g$ for some $g \in
\set{\sum_{j=1}^{m} g_j^* g_j | m\in \mathbb{N}, g_j\in \Real [
 \Fg_\infty]}$.
\end{proposition}

\begin{proof}
Since every II$_1$ factor contains matrix algebras of arbitrary size
we see that $T_n(\Real)$ coincides with the set of correlation
matrices.  The last statement follows from Lemma~\ref{reallemma}.
\end{proof}

\begin{corollary}
If quadratic  $f\in \F$, $f(u_1, \ldots, u_n)=
  \alpha + \sum_{i\not= j }\alpha_{i j} u_i^* u_j$ is such that
  $$\Tr (f(U_1, \ldots, U_n)) = 0 $$ for all unitary matrices $U_1,\ldots, U_n$
  then $f=0$.
\end{corollary}
\begin{proof}
For every $k\not= j$ and $t\in [0,1]$ the matrix $P_1 = I +
(E_{kj}+E_{jk}) t$ is a real correlation matrix. Hence by the theorem there
are unitary matrices $U_1, \ldots , U_n$ such that  $P_1 =
[\tr(U_t^*U_s)]_{ts}$. Then the matrix $P_2 = I + (i E_{kj}- i
E_{jk})t$ is equal to $[\tr(V_t^*V_s)]_{ts}$ where $V_t = U_t$ for
$t\not=j$ and $V_j = i U_j$ are  unitary matrices. 
Hence  $\alpha +(\alpha_{kj}+
\alpha_{jk})t=0$ and $\alpha + (\alpha_{kj}-  \alpha_{jk}) i t=0$.
From which  follows that $\alpha =\alpha_{kj}=0$ and hence $f=0$.
\end{proof}


\end{document}